\documentclass[reqno]{amsart}
\usepackage{bbold}
\usepackage{hyperref,color}
\usepackage[latin1]{inputenc}
\usepackage{epsfig}
\usepackage{amssymb}
\usepackage{dsfont}
\usepackage{gensymb}
\usepackage{upgreek}
\usepackage{color}
\usepackage{amssymb}
\usepackage{dsfont}
\usepackage{amsthm}
\usepackage{amsmath}
\usepackage{amsfonts}
\usepackage{amssymb}
\usepackage{cancel}
\usepackage{nicefrac}
\usepackage{graphicx}
\usepackage{epsfig}
\usepackage{floatrow}
\usepackage{lmodern}
\usepackage{ae}

\newtheorem{theorem}{Theorem}[section]

\newtheorem{lemma}{Lemma}[section]

\begin{document}
	
	\title[non-degenrate Kirchhoff system  via heat conduction]{Well-posedness and  direct internal stability of  coupled non-degenrate Kirchhoff system  via heat conduction}
	
	\author{Akram Ben Aissa*}

\begin{abstract}
In the paper under study, we consider the following coupled non-degenerate Kirchhoff system
\begin{equation}\label{P}
\left \{
\begin{aligned}
&\displaystyle y_{tt}-\upvarphi\Big(\int_\Omega | \nabla y |^2\,dx\Big)\Delta y +\upalpha \Delta \uptheta=0, &\mbox{ in }&\; \Omega \times (0, +\infty)\\
&\displaystyle \uptheta_t-\Delta \uptheta-\upbeta \Delta y_t =0, &\mbox{ in }&\; \Omega \times (0, +\infty)\\
&\displaystyle y=\uptheta=0,\; &\mbox{ on }&\;\partial\Omega\times(0, +\infty)\\
&\displaystyle y(\cdot, 0)=y_0, \; y_t(\cdot, 0)=y_1,\;\uptheta(\cdot, 0)=\uptheta_0, \;  \; &\mbox{ in }&\; \Omega\\
\end{aligned} \right.
\end{equation}
where $\Omega$ is a bounded open subset of $\mathbb{R}^n$,  $\upalpha$ and $\upbeta$ be two
nonzero real numbers with the same sign and $\upvarphi$ is given by $\upvarphi(s)= \mathfrak{m}_0+\mathfrak{m}_1s$ with some positive constants $\mathfrak{m}_0$ and $\mathfrak{m}_1$.
 So we prove existence of solution and establish its exponential decay. The method used is based on  multiplier technique and some integral inequalities due to Haraux and Komornik\cite{H1,KOM}.
  \end{abstract}

\keywords{Well-posedness, exponential decay, multiplier method, internal stability,  non-degenrate Kirchhoff system, heat conduction\\
\\
*:UR Analysis  and Control  of PDE's, UR 13ES64\\
Higher Institute of transport and Logistics of Sousse, University of Sousse,  Tunisia.\\
email:akram.benaissa@fsm.rnu.tn}

\maketitle

\section{Introduction}

In these last few years,  Kirchhoff-type equations with non-linear or lineair  internal feedback and source term have been  studied by many authors.\\
For instance, the primary equation due to Kirchhoff is
\begin{equation}\label{kirg}
\rho hu_{tt}-\Bigg\{p_0+\frac{\varepsilon h}{2L}\int_0^L|u_x|^2dx\Bigg\}u_{xx}+\delta u_t+f(x,u)=0
\end{equation}
for $t\geq 0$ and $0<x<L$, where $u=u(t, x)$ is the lateral displacement at the time $t$
and at the space coordinate $x$, $\varepsilon$  the Young modulus, $\rho$  the mass density, $h$ the crosssectional area, $L$ the length of the string, $p_0$ the initial axial tension, $\delta$ the resistance modulus, and $f$ the external force. When $\delta =f = 0$, Eq.(\ref{kirg})  was introduced by Kirchhoff in \cite{kirch}. Further details and physical phenomena described by Kirchhoff's classical theory can be found in \cite{vill}.\\
Let us first review some known results on analogous problems. So, following  Nishihara and Yamada    {\bf\cite{N1}}, Ono {\bf\cite{O}}, they established  existence of global solutions for small data, decay property of the energy and blow-up of solutions. In addition,  degenerate or nondegenerate Kirchhoff equation with weak dissipation is in the following form
 $$
 y_{tt}-\upvarphi\Big(\int_\Omega | \nabla y |^2\,dx\Big)\Delta y +\sigma(t)g(y_t)=0.
 $$ 
 Benaissa et al.  {\bf\cite{B}} used the multiplier method and general weighted integral inequalities to estimate the whole energy of such system. \\
Lasiecka et al.  {\bf\cite{La2}} studied the existence and exponential stability of solutions to a quasilinear system arising in the modeling of nonlinear thermoelastic plates.\\
Also, Lasiecka et al.  {\bf\cite{La1}} considered the thermoelastic Kirchhoff-Love plate, they studied  the local well-posedness, so they proved that  unique classical local solution is extended globally, provided the initial data are sufficiently small at the lowest energy level  and an exponential decay rate is further stated.\\
Tebou  {\bf\cite{La2}} considered
\begin{equation}\label{ts}
\left \{
\begin{aligned}
&y_{tt}-c^2\Delta y +\upalpha(-\Delta)^\mu \uptheta=0,\\
& \uptheta_t-\nu \Delta \uptheta-\upbeta  y_t =0.\\
\end{aligned} \right.
\end{equation}
He showed that  the associated semigroup is not  stable (uniformly) for the values of $\mu\in[0,1]$. Hence, he  proposed an explicit non-uniform decay rate.
Afterwards, for $\mu=1$,  system (\ref{ts}) was discussed by Lebeau and Zuazua  {\bf\cite{G}} and subsequently by Albano and Tataru {\bf\cite{A}}.  So in the same paper {\bf\cite{La2}},  Tebou showed  that the corresponding semigroup is exponentially stable but not analytic.\\
In other context, Tebou et al.  {\bf\cite{L2}} investigated   a thermoelastic plate with rotational forces as
\begin{equation*}
\left \{
\begin{aligned}
&y_{tt}+(-\Delta)^\mu y_{tt}+\Delta^2 y +\upalpha \Delta \uptheta=0,\\
& \uptheta_t-\nu \Delta \uptheta-\upbeta  \Delta y_t =0.\\
\end{aligned} \right.
\end{equation*}
They  showed that, for every $\delta > (2-\mu)/(2-4\mu)$ and for both  clamped and hinged boundary conditions,  the corresponding semigroup is of Gevrey class  when the parameter $\delta$ lies in the interval $(0,1/2)$. Then, they  obtained exponential decay for the  associated semigroup  for hinged boundary conditions, when $\mu$ lies in $(0,1]$. At the end, they  ensured, by constructing a counterexample, that,
under hinged boundary conditions, the semigroup is not analytic, for all $\mu$ in $(0,1]$.\\

The rest of the paper is structured as follows. Besides  the present introduction, section 2 is devoted to state our main results  concerning global well-posedness as well as exponential
stability of Eqs. (\ref{P}). In, Section 3 and 4, we prove our main results.\\\\

\textit{\bf Conceptualization and Methodology}:\\
The purpose of this study is to construct a stability theory under suitable conditions for system (\ref{P}) and apply it to specific physical and mechanical engineering models. In certain instances, exponential stability with
respect to the state space energy can be readily derived using Lyapunov, energy and spectral methods.\\
The energy method is a popular strategy in showing stability of systems defined in the entire space.
 However, employing the energy method to some physical systems on bounded domains necessitates additional regularity and
 compatibility conditions on the data. We emphasize here  that significant feature and difficulty of the problem fits in the quasilinearity appearing in Eq.(\ref{P}) which is topologically hard since the nonlinear coefficient depends only on the time component.

\section{Well-posedness and energey decay}

The  main  results of the paper reads as follows.
\begin{theorem}\label{TH1}(Well-posedness).
 \label{ch5th}
 Let $ (y_0,y_1) \in H^2 (\Omega)\cap H_0^1 (\Omega)\times H_0^1 (\Omega) $, \, $ \uptheta_0\in H_0^1(\Omega)$ and assume that  $\{y_0,y_1,\uptheta_0\}$ are small enough. Then the problem (\ref{P}) has a unique weak solution $ (y,y_t,\uptheta) $ such that for any $ T > 0$, we have
 $$ 
 (y,y_t) \in L^{\infty} ( 0, T; H^2 (\Omega)\cap H_0^1 (\Omega) ) \times L^{\infty} ( 0, T;  H_0^1 (\Omega)) ,
 $$
  $$
 \uptheta \in L^\infty(0,T;H_0^1(\Omega)). 
 $$
 \end{theorem}
\begin{theorem}\label{TH2}(Exponential stability.) Let $(y,y_t,\uptheta)$ be the solution of (\ref{P}). Then the
energy functional (\ref{EN}) satisfies 
$$
\mathds{E}(t)\leq C\mathds{E}(0)e^{-\omega t},\,\,\, \forall t \geq 0
$$
where $C$ and $\omega$ are positive constants independent of the initial data.
\end{theorem}
Let us now introduce the  energy functional associated to (\ref{P}) which is given by
\begin{equation}\label{EN}
\mathds{E}(t)=\frac{1}{2}\int_\Omega |y_t|^2\,dx+\frac{\mathfrak{m}_0}{2}\int_\Omega |\nabla y|^2\,dx+\frac{\mathfrak{m}_1}{4}\Big(\int_\Omega |\nabla y|^2\,dx\Big)^2+\frac{\upalpha}{2\upbeta}\int_\Omega |\uptheta|^2\,dx,\,\,\,\ \forall t\geq 0.
\end{equation}
So, as a first result of this paper, we have the following.
\begin{lemma}\label{tassou}
Let $(y,y_t,\uptheta)$ be a solution to the problem (\ref{P}). Then, the energy functional
defined by (\ref{EN}) satisfies
\begin{equation}\label{L3}
\mathds{E}'(t)=-\frac{\upalpha }{\upbeta} \int_\Omega |\nabla \uptheta|^2\,dx \leq 0,\,\,\,\,\, \forall t\geq 0.
\end{equation}
That is, the  energy functional is a nonincreasing function.
\end{lemma}
\begin{proof}
Integrating by parts the first equation of (\ref{P}) after multiplying it by $y_t$,\\
yielding
\begin{equation}\label{H1g}
\frac{1}{2}\frac{d}{dt}\int_\Omega |y_t|^2\,dx +\frac{\mathfrak{m}_0}{2}\frac{d}{dt}\int_\Omega  |\nabla y|^2\,dx +\frac{\mathfrak{m}_1}{4}\frac{d}{dt}\Big(\int_\Omega  |\nabla y|^2\,dx\Big)^2+\upalpha \int_\Omega  \nabla \uptheta \nabla y_t\,dx=0.
\end{equation}
Afterwards, as previous, integrating the second  equation of (\ref{P}) over $\Omega$ after  multiplying it by $ \uptheta$, we obtain
\begin{equation}\label{H2}
\frac{1}{2\upbeta }\frac{d}{dt} \int_\Omega |\uptheta|^2\,dx + \frac{1}{\upbeta}\int_\Omega  |\nabla \uptheta|^2\,dx= \int_\Omega  \nabla \uptheta \nabla y_t\,dx.
\end{equation}
Inserting (\ref{H2}) into (\ref{H1g}), we get
\begin{eqnarray*}
&&\frac{1}{2}\frac{d}{dt}\int_\Omega |y_t|^2\,dx + \frac{\mathfrak{m}_0}{2}\frac{d}{dt}\int_\Omega  |\nabla y|^2\,dx +\frac{\mathfrak{m}_1}{4}\frac{d}{dt}\Big(\int_\Omega  |\nabla y|^2\,dx\Big)^2 +
\frac{\upalpha}{2\upbeta }\frac{d}{dt} \int_\Omega |\uptheta|^2\,dx\\
&&=-\frac{\upalpha }{\upbeta} \int_\Omega |\nabla \uptheta|^2\,dx.
\end{eqnarray*}
The proof of Lemma \ref{tassou} is thus complete.
\end{proof}
\section{Proof of  Theorem \ref{TH1} }
As a powerful tool to prove the existence of a global solutions for problem (\ref{P}) is the Faedo-Galerkin method. In fact, let $(e_k)_{k\in\mathbb{N} }$ be normalized eigenfunctions of the negative Laplacian with Dirichlet
boundary conditions
\begin{equation*}
\left\{
\begin{aligned}
& -\Delta e_k=\lambda_k e_k,\,\,\,\,\mbox{in}\,\,\, \Omega & \\
 &e_k=0,\,\,\,\,\mbox{in}\,\,\, \partial \Omega.&
 \\
\end{aligned}
\right.
\end{equation*}
Then, the family  $\{e_k | k\in\mathbb{N} \}$  forms an orthonormal basis of $L^2(\Omega)$. Furthermore, we consider  $V^n=span\{e_m| m=1,2,\ldots,n\}$. So here, several steps are envolved.\\
$\blacktriangleright$\textbf{ Step 1:} 
We construct approximate solutions $(y^n,y^n_t,\uptheta^n),\,\,  n= 1, 2, 3,\ldots$,  in the form
$$ 
y^n(x,t)=\sum_{m=1}^n h^{n}_m(t)e_m(x),
$$
and
$$
\uptheta^n(x,t)=\sum_{m=1}^n c^{n}_m(t)e_m(x)
$$
 where $h^{n}_m,\,\, c^{n}_m\,\, (m= 1, 2,\ldots,n)$ are determined by the following ordinary differential
equations
\small{\begin{equation}\label{PO}
\left\{
\begin{aligned}
&  (y^n_{tt}-\upvarphi\Big(\int_\Omega | \nabla y^n |^2\,dx\Big)\Delta y^n +\upalpha \Delta  \uptheta^n,e_m)=0
 \  &\forall  e_m\in V^n& \\
 &(\uptheta^n_t-\Delta \uptheta^n-\upbeta \Delta y^n_t  ,e_m)=0 \ &\forall e_m \in V^n&
 \\
\end{aligned}
\right.
\end{equation}}
 with initial conditions
\begin{equation}\label{p7}
y^n(x,0)=y_0^n=\sum_{m=1}^n\langle f,e_m\rangle e_m\to y_0,\quad \text{in }
H^2(\Omega) \cap H^1_0 (\Omega) \  \text{ as } n\to \infty,
\end{equation}
\begin{equation}\label{p8}
y_t^n(x,0)=y_1^n=\sum_{m=1}^n\langle f_t,e_m\rangle e_m\to y_1,\quad \text{in }
H^1_0(\Omega)\text{ as } n\to \infty,
\end{equation}
\begin{equation}\label{p72}
\uptheta^n(x,0)=\uptheta_0^n=\sum_{m=1}^n\langle g,e_m\rangle e_m\to \uptheta_0,\quad \text{in }
H^1_0(\Omega)\text{ as } n\to \infty.
\end{equation}
\hspace{0.2mm}\\
The system (\ref{PO})-(\ref{p72}) of ordinary differential equation of variable $t$ admits a solution $(y^n, y_t^n, \uptheta^n)$ on the interval $[0,t_n)$.
At the beginning, we will start to identify some a priori estimates in order to prove that $ t_n = \infty$. After that, we will show that the
sequence of solutions to (\ref{PO}) converges to a solution of (\ref{P}) with the claimed smoothness.\\
 $\blacktriangleright$\textbf{Step 2:} If we  multiply the first  and the second equations of (\ref{PO}) by  $h_{m,t}^n(t)$  and $c^n_m(t)$ respectively and sum over $m$ from $1$ to $n$, we get  
\begin{equation}\label{F1}
\begin{split}
&\int_\Omega |y^n_t|^2\,dx + \Big(\mathfrak{m}_0 +\frac{\mathfrak{m}_1}{2}\int_\Omega  |\nabla y^n|^2\,dx\Big)\int_\Omega  |\nabla y^n|^2\,dx +
\frac{\upalpha}{\upbeta } \int_\Omega |\uptheta^n|^2\,dx+\frac{2\upalpha }{\upbeta}\int_0^t \int_\Omega |\nabla \uptheta^n(s)|^2\,dx\,ds\\&\leq  
\int_\Omega |y^n_1|^2\,dx + \Big(\mathfrak{m}_0 +\frac{\mathfrak{m}_1}{2}\int_\Omega  |\nabla y_0^n|^2\,dx\Big)\int_\Omega  |\nabla y^n_0|^2\,dx +
\frac{\upalpha}{\upbeta } \int_\Omega |\uptheta^n_0|^2\,dx\\&\leq 2\mathds{E}(0),\,\,\,\,\, \,\,\,\,\,\,\,\, \forall t\in [0,t_n).
\end{split}
\end{equation}
Therefore, we deduce  that $t_n=\infty$, and that
\begin{equation}\label{S1}
 y^n\,\,\,\,\, \hbox{is bounded in } \,\,\,\,\,\,\, L^\infty(0,T;H^1_0(\Omega))
\end{equation}
\begin{equation}\label{S2}
 y^n_t\,\,\,\,\, \hbox{is bounded in } \,\,\,\,\,\,\, L^\infty(0,T;L^2(\Omega))
\end{equation}
\begin{equation}\label{S3}
 \uptheta^n\,\,\,\,\, \hbox{is bounded in } \,\,\,\,\,\,\, L^\infty(0,T;L^2(\Omega))\cap L^2(0,T;H_0^1(\Omega)).
 \end{equation}
 $\blacktriangleright$\textbf{ Step 3:}  Replacing $e_m$ by $-\Delta e_m$, and doing in the same  manner as previous step, we get
 \begin{align*}
&\frac{d}{dt}\Big[\int_\Omega |\nabla y^n_t|^2\,dx + \Big(\mathfrak{m}_0 +\frac{\mathfrak{m}_1}{2}\int_\Omega  |\nabla y^n|^2\,dx\Big)\int_\Omega  |\Delta y^n|^2\,dx +\frac{\upalpha}{\upbeta } \int_\Omega |\nabla \uptheta^n|^2\,dx\Big]\\&+\frac{2\upalpha }{\upbeta} \int_\Omega |\Delta \uptheta^n|^2\,dx\\
&=\int_\Omega  |\Delta y^n|^2\,dx \frac{d}{dt}\Big[\mathfrak{m}_0 +\frac{\mathfrak{m}_1}{2}\int_\Omega  |\nabla y^n|^2\,dx\Big]\\
&=\mathfrak{m}_1\Big(\int_\Omega  \nabla y^n \nabla y_t^n\,dx \Big)\int_\Omega  |\Delta y^n|^2\,dx. 
\end{align*}
\hspace{0.2mm}\\\\
Using (\ref{F1}) and Cauchy-Schwarz  inequality  the following  estimate holds
 \begin{equation*}
\begin{split}
&\frac{d}{dt}\Big[\int_\Omega |\nabla y^n_t|^2\,dx + \Big(\mathfrak{m}_0 +\frac{\mathfrak{m}_1}{2}\int_\Omega  |\nabla y^n|^2\,dx\Big)\int_\Omega  |\Delta y^n|^2\,dx +
\frac{\upalpha}{\upbeta } \int_\Omega |\nabla \uptheta^n|^2\,dx\Big]\\&+\frac{2\upalpha }{\upbeta} \int_\Omega |\Delta \uptheta^n|^2\,dx\\&\leq
C \Big(\int_\Omega  |\nabla y_t^n|^2\,dx\Big)^{\frac{1}{2}}\int_\Omega  |\Delta y^n|^2\,dx.
\end{split}
\end{equation*}
Integrating the last inequality over $(0, t)$, we get
 \begin{equation}\label{F12}
\begin{split}
\mathds{E}_*^n(t) +\frac{2\upalpha }{\upbeta} \int_0^t\int_\Omega |\Delta \uptheta^n(s)|^2\,dx\,ds\leq \mathds{E}_*^n(0) +C\int_0^t (\mathds{E}_*^n(s))^{\frac{3}{2}}\,ds
\end{split}
\end{equation}
where
$$
\mathds{E}_*^n(t)=\int_\Omega |\nabla y^n_t|^2\,dx + \Big(\mathfrak{m}_0 +\frac{\mathfrak{m}_1}{2}\int_\Omega  |\nabla y^n|^2\,dx\Big)\int_\Omega  |\Delta y^n|^2\,dx +
\frac{\upalpha}{\upbeta } \int_\Omega |\nabla \uptheta^n|^2\,dx.
$$
To complete this step, we need the following Lemma.
\begin{lemma}\label{ch5LM2}(Modified Gronwall inequality)
Let $G$ and $f$ be non-negative functions on $[0,+\infty)$  satisfying
$$
0\leq G(t)\leq  K+\int_0^t f(s)G(s)^{r+1}\, ds,
$$
with $K>0$ and $r>0$. Then
$$
G(t)\leq \Big\{K^{-r}-r\int_0^t f(s)\, ds\Big\}^{-1/r},
$$
as long as the RHS exists.
\end{lemma}
\smallskip
So, an immediate application of this lemma with 
$$G(t)=\mathds{E}_*^n(t),\;\;K(t)=\mathds{E}_*^n(0)\quad\text{and}\; f(t)=C$$
gives us
  $$
  \mathds{E}_*^{n}(t) \leq \Big\{\Big(\mathds{E}_*^{n}(0)\Big)^{-\frac{1}{2}} -C\frac{1}{2}\int_0^t\,ds\Big\}^{-2}.
  $$
  Therefore, if initial data $\{u_0,u_1\}$  are  sufficiently small, we deduce that
  $$
  \mathds{E}_*^{n}(t) \leq \Big\{\Big(\mathds{E}_*(0)\Big)^{-\frac{1}{2}} -\frac{CT}{2}\Big\}^{-2}..
  $$
    Hence, we conclude that
\begin{gather}
 \label{ch5S33}
y^n_t \,\,\,\,\, \hbox{is bounded in } \,\,\,\,\,\,\, L^{\infty}(0, T, H_{0}^{1}(\Omega)), \\
\label{ch5S35}
  \Big(\mathfrak{m}_0 +\frac{\mathfrak{m}_1}{2}\int_\Omega  |\nabla y^n|^2\,dx\Big)\Delta y^{n}   \,\,\,\,\, \hbox{is bounded in } \,\,\,\,\,\,\,L^\infty (0, T, L^2(\Omega)),
\end{gather}
and
\begin{equation} \label{SZ1}
\Delta y^{n}     \text{ is bounded in }L^\infty (0, T, L^2(\Omega)),
\end{equation}
\begin{equation} \label{SZ2}
\nabla \uptheta^n    \text{ is bounded in }L^\infty (0, T, L^2(\Omega)).
\end{equation}
$\blacktriangleright$\textbf{ Step 4:} {\it Passing to the limit:}\\
Applying Dunford-Pettis and Banach-Alaoglu-Bourbaki theorems, we conclude from 
(\ref{S1})-(\ref{S3}), (\ref{SZ2}) and (\ref{SZ1}) that there exists a subsequence $\{y^m,\uptheta^m\}$
of $\{y^n,\uptheta^n\}$ such that
\begin{equation}\label{c1}
y^m\rightharpoonup^* y,\,\,\,\,\, \hbox{ in}\,\,\,\,\,\,\, L^\infty(0, T; H^2(\Omega) \cap H^1_0(\Omega)) 
\end{equation}
\begin{equation}\label{c2}
y^m_t \rightharpoonup^* y_t,\,\,\,\,\, \hbox{in}\,\,\,\,\,\,\,L^\infty(0, T; H_0^1(\Omega)) 
\end{equation}
\begin{equation}\label{c3}
\Big(\int_\Omega | \nabla y^m |^2\,dx\Big)\Delta  y^m\rightharpoonup^* \upchi ,\,\,\,\,\, \hbox{ in}\,\,\,\,\,\,\, L^\infty(0, T; L^2(\Omega)) 
\end{equation}
\begin{equation}\label{c4}
\uptheta^m \rightharpoonup^* \uptheta ,\,\,\,\,\,\hbox{in}\,\,\,\,\,\,\,L^\infty(0, T; H_0^1(\Omega) )
\end{equation}
\begin{equation}\label{c5}
\uptheta^m \rightharpoonup^* \uptheta ,\,\,\,\,\,\hbox{in}\,\,\,\,\,\,\,L^2(0, T; H_0^1(\Omega) ).
\end{equation}
By (\ref{c2}) and (\ref{c4}), we have
\begin{equation}\label{c6}
\Delta^{-1}\uptheta_t^m \rightharpoonup^* \Delta^{-1}\uptheta_t,\,\,\,\,\,\hbox{in}\,\,\,\,\,\,\,L^2(0, T; L^2(\Omega) ),
\end{equation}
where $\Delta^{-1}$ denotes the inverse of the Laplacian with zero Dirichlet boundary
conditions.\\
 We shall prove that, in fact, $\upchi= \Big(\int_\Omega | \nabla y |^2\,dx\Big)\Delta y$, i.e.
 \begin{equation}\label{B1}
 \Big(\int_\Omega | \nabla y^m |^2\,dx\Big) \Delta y^m
\rightharpoonup^* \Big(\int_\Omega | \nabla y |^2\,dx\Big) \Delta y, \hbox{   in  }L^\infty(0, T; L^2(\Omega)).
\end{equation}
As $(y^n)$ is bounded in $L^\infty(0, T, H^2(\Omega)\cap H_0^1(\Omega))$   (by (\ref{SZ1}))
and the embedding of $H^2(\Omega)$ in $L^2(\Omega)$
is compact, we have
\begin{equation}\label{c9}
y^m \longrightarrow y, \hbox{   strongly  in    } L^2(0, T; L^2(\Omega)).
\end{equation}
On the other hand, for $\upsilon\in L^2(0, T; L^2(\Omega))$, we have
\begin{equation}\label{c7}
\begin{split}
&\int_0^T \left(\upchi -\left(\int_\Omega | \nabla y^m |^2\,dx\right) \Delta y, \upsilon\right)\, dt \\&= \int_0^T \left(\upchi -\left(\int_\Omega | \nabla y^m |^2\,dx\right)\Delta y^m, \upsilon\right)\, dt+
\int_0^T \left(\int_\Omega | \nabla y^m |^2\,dx\right) \left(\Delta y^m-\Delta y, \upsilon\right)\, dt\\
&+\int_0^T \Big(\int_\Omega (| \nabla y^m |^2- | \nabla y |^2)\,dx \Big) (\Delta y^m, \upsilon)\, dt.
\end{split}
\end{equation}
We deduce from (\ref{c1}) and (\ref{c3})  that the first and the second terms in
(\ref{c7}) tend to zero as $m\rightarrow \infty$. For the third term, using (\ref{S1}) and (\ref{SZ1}),
we can write (with $c$ positive constant)
\begin{equation*}
\begin{split}
&\int_0^T \left(\int_\Omega | \nabla y^m|^2 -|\nabla y|^2 \right) \left(\Delta y^m, \upsilon\right)\, dt\\&\leq 
c \int_0^T \Big(\int_\Omega | \nabla y^m -\nabla y|^2\,dx\Big)^{\frac{1}{2}}\Big(\int_\Omega | \nabla y^m +\nabla y|^2\,dx\Big)^{\frac{1}{2}}
\Big(\int_\Omega|\Delta y^m|^2\,dx\Big)^{\frac{1}{2}} \Big(\int_\Omega |\upsilon|^2\,dx\Big)^{\frac{1}{2}} dt\\&
\leq c \Big(\int_\Omega | \nabla y^m -\nabla y|^2\,dx\Big)^{\frac{1}{2}}\Big(\int_\Omega |\upsilon|^2\,dx\Big)^{\frac{1}{2}}\,dt.
\end{split}
\end{equation*}
Hence we deduce (\ref{B1}) from (\ref{c9}).\\
Furthermore, using (\ref{c1}), (\ref{c4}) and (\ref{B1}), we have
\begin{equation*}
\begin{split}
&\int_0^T\int_\Omega \left(y_{tt}^n-\left(\mathfrak{m}_0 +\frac{\mathfrak{m}_1}{2}\int_\Omega  |\nabla y^n|^2\,dx\right)\Delta y^{n}+\upalpha \Delta  \uptheta^n,\upsilon\right)\,dx\,dt  \\&\rightarrow
\int_0^T\int_\Omega \left(y_{tt}-\left(\mathfrak{m}_0 +\frac{\mathfrak{m}_1}{2}\int_\Omega  |\nabla y|^2\,dx\right)\Delta y+\upalpha \Delta  \uptheta,\upsilon\right)\,dx\,dt,\,\,\,\,\, \forall \upsilon\in L^2(0, T; L^2(\Omega)).
\end{split}
\end{equation*}
On the other hand, using (\ref{c2}), (\ref{c4}) and (\ref{c6}), we have
\begin{equation*}
\begin{split}
&\int_0^T\int_\Omega \left(\uptheta_{t}^n-\Delta \uptheta^{n}-\upbeta \Delta  y_t^n,\uppsi\right)\,dx\,dt\\& =\int_0^T\int_\Omega \left(\Delta^{-1}\uptheta_{t}^n- \uptheta^{n}-\upbeta   y_t^n,\Delta^{-1}\uppsi\right)\,dx\,dt 
\\&\rightarrow
\int_0^T\int_\Omega \left(\Delta^{-1}\uptheta_{t}- \uptheta-\upbeta y_t,\Delta^{-1}\uppsi\right)\,dx\,dt,\,\,\,\,\, \forall \Delta^{-1}\uppsi\in L^2(0, T; L^2(\Omega)).
\end{split}
\end{equation*}
 $\blacktriangleright$\textbf{ Step 5:} {\it  Proof of uniqueness:}\\
 Let $(y_1,y_{1,t},\uptheta_1)$ and $(y_2,y_{2,t},\uptheta_2)$ be   solutions of
(\ref{P}) with the same initial data,\\
 and setting $Y=y_1-y_2$ and $\uptheta=\uptheta_1-\uptheta_2$. Hence,  we have
\begin{equation*}
\left \{
\begin{aligned}
&\displaystyle Y_{tt}-(\mathfrak{m}_0+\frac{\mathfrak{m}_1}{2}\int_\Omega | \nabla y_1 |^2\,dx)\Delta y_1+(\mathfrak{m}_0+\frac{\mathfrak{m}_1}{2}\int_\Omega | \nabla y_2 |^2\,dx)\Delta y_2
 +\upalpha \Delta  \uptheta=0, \\
&\displaystyle \uptheta_t-\Delta \uptheta-\upbeta \Delta Y_t =0, \\
\end{aligned} \right.
\end{equation*}
with $Y=\uptheta=0$ on $[0,+\infty)\times \partial\Omega$ and $Y(0)=Y_t(0)=\uptheta(0)=0$ in $\Omega$.
\begin{equation}\label{PM}
\left \{
\begin{aligned}
&\displaystyle Y_{tt}-(\mathfrak{m}_0+\frac{\mathfrak{m}_1}{2}\int_\Omega | \nabla y_1 |^2\,dx)\Delta Y-(\mathfrak{m}_0+\frac{\mathfrak{m}_1}{2}\int_\Omega | \nabla y_1 |^2\,dx)\Delta y_2
\\&+(\mathfrak{m}_0+\frac{\mathfrak{m}_1}{2}\int_\Omega | \nabla y_2 |^2\,dx)\Delta y_2
 +\upalpha \Delta  \uptheta=0, \\
&\displaystyle \uptheta_t-\Delta \uptheta-\upbeta \Delta Y_t =0. \\
\end{aligned} \right.
\end{equation}
Taking the $L^2(\Omega)$ inner product of first and second equation of (\ref{PM}) with  $Y_t$ and $\uptheta$ respectively, we get
\begin{equation*}
\begin{split}
\frac{d}{dt} &\Big[\int_\Omega |Y_{t} |^2\,dx +(\mathfrak{m}_0+\frac{\mathfrak{m}_1}{2} \int_\Omega | \nabla y_1|^2\,dx )\int_\Omega |\nabla Y|^2\,dx
+\frac{\upalpha}{\upbeta}\int_\Omega|\uptheta |^2\,dx\Big ]  +\frac{2\upalpha}{\upbeta}\int_\Omega|\nabla \uptheta |^2\,dx
\\&= \frac{d}{dt}(\mathfrak{m}_0+\frac{\mathfrak{m}_1}{2}\int_\Omega |\nabla y_1|^2\,dx)\int_\Omega |\nabla Y |^2\,dx 
\\&+ \int_\Omega\Big\{(\mathfrak{m}_0+\mathfrak{m}_1\int_\Omega |\nabla y_1|^2\,dx)\Delta y_2-(\mathfrak{m}_0+\mathfrak{m}_1\int_\Omega |\nabla y_2|^2\,dx)\Delta y_2\Big\} Y_t\,dx.
 \end{split}
\end{equation*}
Using Cauchy-Schwarz and Young's inequalities , we have 
\begin{equation*}
\begin{split}
\frac{d}{dt} &\Big[\int_\Omega |Y_{t} |^2\,dx +(\mathfrak{m}_0+\frac{\mathfrak{m}_1}{2} \int_\Omega | \nabla y_1|^2\,dx )\int_\Omega |\nabla Y|^2\,dx
+\frac{\upalpha}{\upbeta}\int_\Omega|\uptheta |^2\,dx\Big ]  +\frac{2\upalpha}{\upbeta}\int_\Omega|\nabla \uptheta |^2\,dx
\\&\leq  \frac{d}{dt}(\mathfrak{m}_0+\frac{\mathfrak{m}_1}{2}\int_\Omega |\nabla y_1|^2\,dx)\int_\Omega |\nabla Y |^2\,dx 
+2\mathfrak{m}_1\int_\Omega \{|\nabla y_1|^2- |\nabla y_2|^2\}\,dx  \int_\Omega \Delta y_2Y_t\,dx
\\&\leq 2\mathfrak{m}_1 \int_\Omega |\nabla Y |^2\,dx \int_\Omega \nabla y_1\nabla y_{1t}\,dx\\& +
2\mathfrak{m}_1 \Big(\int_\Omega \{|\nabla y_1- \nabla y_2|^2\}\,dx \Big)^{\frac{1}{2}} 
\Big(\int_\Omega |\Delta y_2|^2\,dx\Big)^{\frac{1}{2}} \Big(\int_\Omega |Y_t|^2\,dx\Big)^{\frac{1}{2}}
\\&\leq \mathfrak{m}_1 \int_\Omega |\nabla Y |^2\,dx \Big(\int_\Omega |\nabla y_1|^2\,dx+\int_\Omega |\nabla y_{1t}|^2\,dx\Big)\\& +
2\mathfrak{m}_1 \int_\Omega |\nabla Y|^2\,dx 
\int_\Omega |\Delta y_2|^2\,dx +\int_\Omega |Y_t|^2\,dx.
 \end{split}
\end{equation*}
 Integrating  it over $(0,t)$, we conclude that
\begin{equation*}
\begin{split}
&\int_\Omega |Y_{t}|^2\,dx +(\mathfrak{m}_0+\frac{\mathfrak{m}_1}{2} \int_\Omega | \nabla y_1|^2\,dx) \int_\Omega |\nabla Y|^2\,dx+\frac{\upalpha}{\upbeta}\int_\Omega|\uptheta |^2\,dx\\& \leq C
 \int_0^t \Big\{\int_\Omega |Y_{t}(s)|^2\,dx +\mathfrak{m}_1 \int_\Omega | \nabla y_1(s)|^2\,dx \int_\Omega |\nabla Y(s)|^2\,dx+\frac{\upalpha}{\upbeta}\int_\Omega|\uptheta(s) |^2\,dx \Big\}\,ds.
 \end{split}
\end{equation*}
which, by  Gronwall's lemma, implies $Y \equiv 0$ and $\uptheta=0$. 
The proof of Theorem  \ref{ch5th} is now completed.

\section{Proof of Theorem \ref{TH2}}
In this section, we  prove our stability result for the energy of the
solution of system (\ref{P}), using the multiplier technique.
This proof will be established in three  steps and needed the following Lemma due   to Martinez \cite{Ma}.
\begin{lemma}\label{H1}
 Let $\mathds{E}:\mathbb{R}^+ \longrightarrow \mathbb{R}^+$ be a non-increasing function and assume that
there are two constants 
 $\mu\geq 0$ and $\omega>0$ such that
 $$
 \int_t^{+\infty} \mathds{E}(s)^{\mu+1}\,ds\leq  \omega \mathds{E}(0)^\mu \mathds{E}(t),\,\,\,\,\,  \forall t \geq  0.
 $$
Then, we have for every $t>0$
\begin{equation*}
\left\{\begin{array}{ccc}
\mathds{E}(t)\leq  \mathds{E}(0)\Big(\frac{1+\mu}{1+\omega \mu t}\Big)^{-\frac{1}{\mu}}, &\text{if}\ \mu> 0\\
\\
\mathds{E}(t)\leq \mathds{E}(0)e^{1-\omega t}, &\text{if}\ \mu= 0.\\
\end{array}\right.
\end{equation*}
\end{lemma}
\bigskip
$\blacktriangleright$\textbf{ Step 1:} 
Let $\mu\geq 0$ be a non-negative constant. Multiplying the first equation of (\ref{P}) by $ \mathds{E}^\mu y$, and integrating over $\Omega \times (S,T)$ , we find
$$
0= \int_S^T \mathds{E}^\mu\int_\Omega  y\left(y_{tt}-\upvarphi\Big(\int_\Omega | \nabla y |^2\,dx\Big)\Delta y +\upalpha \Delta  \uptheta\right)\,dx\,dt.
$$
Using Green's formula, we derive
\begin{equation*}\label{R2}
\begin{split}
0&= \Big[ \mathds{E}^\mu \int_\Omega   y y_t \,dx\Big]_S^T- \int_S^T\mathds{E}^\mu \int_\Omega   |y_t|^2 \,dx\,dt
\\&-  \mu \int_S^T \mathds{E}^{\mu-1}\mathds{E}'\int_\Omega y y_t\,dx\,dt+\int_S^T\mathds{E}^\mu \upvarphi\Big(\int_\Omega | \nabla y |^2\,dx\Big)  \int_\Omega   |\nabla y|^2 \,dx\,dt
 \\&-\upalpha   \int_S^T\mathds{E}^\mu \int_\Omega   \nabla y \nabla \uptheta \,dx\,dt.
\end{split}
\end{equation*}
Using (\ref{EN}), we have 
\begin{equation}\label{G1L}
\begin{split}
2\int_S^T\mathds{E}^{\mu+1} \,dt
& =-\Big[ \mathds{E}^\mu \int_\Omega   y y_t \,dx\Big]_S^T+2 \int_S^T\mathds{E}^\mu \int_\Omega   y_t^2 \,dx\,dt
\\&+  \mu \int_S^T \mathds{E}^{\mu-1}\mathds{E}'\int_\Omega y y_t\,dx\,dt +\upalpha   \int_S^T\mathds{E}^\mu \int_\Omega   \nabla y \nabla \uptheta \,dx\,dt\\&
+\frac{\upalpha}{\upbeta}\int_S^T\mathds{E}^\mu \int_\Omega \uptheta^2\,dx\,dt-\frac{\mathfrak{m}_1}{2} \int_S^T\mathds{E}^{\mu} \Big(\int_\Omega | \nabla y |^2\,dx\Big)^2\,dt.
\end{split}
\end{equation}
Since $\mathds{E}$ is nonincreasing, and using Cauchy-Schwarz and Poincar\'e inequalities, we have
\begin{equation}\label{A1}
\Big| \Big[ \mathds{E}^\mu \int_\Omega   y y_t \,dx\Big]_S^T\Big|\leq C_0\mathds{E}^{\mu+1}(S)
\end{equation}
\begin{equation}\label{A2}
\begin{split}
\Big|\mu \int_S^T \mathds{E}^{\mu-1}\mathds{E}'\int_\Omega y y_t\,dx\,dt\Big| &\leq C_0 \int_S^T \mathds{E}^{\mu}(-\mathds{E}')\,dt 
\\&\leq C_0 \mathds{E}^{\mu+1}(S) 
\end{split}
\end{equation}
\begin{equation}\label{A3}
\begin{split}
\frac{\upalpha}{\upbeta}\int_S^T\mathds{E}^\mu \int_\Omega \uptheta^2\,dx\,dt &\leq C_0 \int_S^T \mathds{E}^{\mu}(-\mathds{E}')\,dt 
\\&\leq C_0 \mathds{E}^{\mu+1}(S) 
\end{split}
\end{equation}
\begin{equation*}
\begin{split}
\upalpha   \int_S^T\mathds{E}^\mu \int_\Omega   \nabla y \nabla \uptheta \,dx\,dt\leq C_0 \int_S^T\mathds{E}^{\mu+\frac{1}{2}}(-\mathds{E}')^{\frac{1}{2}}\,dt.
\end{split}
\end{equation*}
Now, fix an arbitrarily small $\upepsilon_0 > 0$, and applying Young's inequality, we obtain
\begin{equation}\label{A4}
\begin{split}
\Big|\upalpha   \int_S^T\mathds{E}^\mu \int_\Omega   \nabla y \nabla \uptheta \,dx\,dt \Big| &\leq 
 \upepsilon_0 \int_S^T \mathds{E}^{2\mu+1}(t)\,dt+ C(\upepsilon_0) \int_S^T (-\mathds{E}')\,dt
\end{split}
\end{equation}
Taking into account (\ref{A1})-(\ref{A4}) into (\ref{G1L}) and Poincar\'e inequality, we obtain
\begin{equation}\label{G15}
\begin{split}
2\int_S^T\mathds{E}^{\mu+1}\,dt
& \leq \upepsilon_0 \int_S^T \mathds{E}^{2\mu+1}\,dt+C_0\mathds{E}^{\mu+1}(S)+C_0 \mathds{E}(S) \\&+2c_s \int_S^T\mathds{E}^\mu \int_\Omega    | \nabla y_t|^2 \,dx\,dt
-\frac{\mathfrak{m}_1}{2} \int_S^T\mathds{E}^{\mu} \Big(\int_\Omega | \nabla y |^2\,dx\Big)^2\,dt.
\end{split}
\end{equation}
$\blacktriangleright$\textbf{ Step 2:}  In this step, we are going to estimate  the four term  in the right hand side of (\ref{G15}).\\
 Multiplying the second  Eq. of (\ref{P}) by $ \mathds{E}^\mu y_t$, and integrating by parts over $\Omega \times (S,T)$, we get
\begin{equation}\label{G7}
\begin{split}
0&=\int_\Omega\int_S^T\mathds{E}^{\mu} y_t(\uptheta_t-\Delta \uptheta-\upbeta \Delta y_t)\,dx\,dt\\&=
\Big[ \mathds{E}^\mu \int_\Omega    y_t  \uptheta\,dx\Big]_S^T - \mu \int_S^T \mathds{E}^{\mu-1}\mathds{E}' \int_\Omega   y_t \uptheta \,dx\,dt
- \int_S^T \mathds{E}^{\mu}\int_\Omega \uptheta y_{tt}  \,dx\,dt  \\&+\int_S^T \mathds{E}^{\mu}\int_\Omega  \nabla y_t \nabla  \uptheta  \,dx\,dt +\upbeta\int_S^T \mathds{E}^{\mu}\int_\Omega | \nabla y_t |^2\,dx\,dt.
\end{split}
\end{equation}
Since $\mathds{E}$ is nonincreasing, and using Cauchy-Schwarz inequalities, we have
\begin{equation}\label{AG1}
\Big| \Big[ \mathds{E}^\mu \int_\Omega    y_t \uptheta\,dx\Big]_S^T\Big|\leq C_1\mathds{E}^{\mu+1}(S),
\end{equation}
\begin{equation}\label{AG2}
\begin{split}
\Big|\mu \int_S^T \mathds{E}^{\mu-1}\mathds{E}'\int_\Omega  y_t \uptheta\,dx\,dt\Big| &\leq C \int_S^T \mathds{E}^{\mu}(-\mathds{E}')\,dt 
\\&\leq C_1 \mathds{E}^{\mu+1}(S) .
\end{split}
\end{equation}
Applying Young's inequality, we get
\begin{equation}\label{AG3}
\begin{split}
\Big|\int_S^T \mathds{E}^{\mu}\int_\Omega  \nabla y_t \nabla  \uptheta  \,dx\,dt\Big| &\leq \frac{\upbeta}{2} \int_S^T \mathds{E}^{\mu}\int_\Omega | \nabla y_t |^2\,dx\,dt +\frac{1}{2\upbeta} \int_S^T \mathds{E}^{\mu}\int_\Omega | \nabla \uptheta |^2\,dx\,dt\\&\leq \frac{\upbeta}{2} \int_S^T \mathds{E}^{\mu}\int_\Omega | \nabla y_t |^2\,dx\,dt +\frac{1}{2\upalpha}\int_S^T \mathds{E}^{\mu} (-\mathds{E}')\,dt
\\&\leq \frac{\upbeta}{2} \int_S^T \mathds{E}^{\mu}\int_\Omega | \nabla y_t |^2\,dx\,dt +C_1 \mathds{E}^{\mu+1}(S).
\end{split}
\end{equation}
Reporting  (\ref{AG1})-(\ref{AG3}) into (\ref{G7}) yields
\begin{equation}\label{GG7}
\begin{split}
\frac{\upbeta}{2}\int_S^T \mathds{E}^{\mu}\int_\Omega | \nabla y_t |^2\,dx\,dt\leq  C_1\mathds{E}^{\mu+1}(S)+ \int_S^T \mathds{E}^{\mu}\int_\Omega \uptheta y_{tt}  \,dx\,dt.
\end{split}
\end{equation}
$\blacktriangleright$\textbf{ Step 3:}  In this step, we are going to estimate the  second term in the right
hand side of (\ref{GG7}).\\
Multiplying the first equation in (\ref{P}) by $\mathds{E}^\mu  \uptheta$ and using Green's formula over $\Omega \times (S,T)$, we find
\begin{equation}\label{L2}
\int_S^T \mathds{E}^\mu\int_\Omega  \uptheta y_{tt}\,dx\,dt = -\int_S^T \mathds{E}^\mu \upvarphi\Big(\int_\Omega | \nabla y |^2\,dx\Big)\int_\Omega \nabla y\nabla \uptheta\,dx\,dt 
+\upalpha \int_S^T \mathds{E}^\mu\int_\Omega | \nabla \uptheta|^2\,dx\,dt.
\end{equation}
Thanks to Cauchy-Schwarz inequality, and the definition of the energy $\mathds{E}$, we easily derive
\begin{equation}\label{L1}
\begin{split}
&\Big|\int_S^T \mathds{E}^\mu \upvarphi\Big(\int_\Omega | \nabla y |^2\,dx\Big)\int_\Omega \nabla y\nabla \uptheta\,dx\,dt \Big|\\&\leq 
\sqrt{\frac{2}{\mathfrak{m}_0}} \int_S^T  \upvarphi\Big(\int_\Omega | \nabla y |^2\,dx\Big)\mathds{E}^{\mu+\frac{1}{2}}\Big(\int_\Omega |\nabla \uptheta| ^2\,dx\Big)^{\frac{1}{2}}\,dt.
\end{split}
\end{equation}
Now, pick an arbitrarily small $\upepsilon_1 > 0$,  applying Young's inequality and using the expression of $\upvarphi$, we obtain
\begin{equation}\label{L1}
\begin{split}
&\Big|\int_S^T \mathds{E}^\mu \upvarphi\Big(\int_\Omega | \nabla y |^2\,dx\Big)\int_\Omega \nabla y\nabla \uptheta\,dx\,dt \Big|\\&\leq 
2\upepsilon_1 \int_S^T  \mathds{E}^{2\mu+1}\,dt+ \frac{4\mathfrak{m}_1}{\mathfrak{m}_0^2}\upepsilon_1 \int_S^T  \mathds{E}^{2\mu+2}\,dt+C(\upepsilon_1)
 \int_S^T  \upvarphi\Big(\int_\Omega | \nabla y |^2\,dx\Big) (-\mathds{E}')\,dt
 \\&\leq
2\upepsilon_1 \int_S^T  \mathds{E}^{2\mu+1}\,dt+ \frac{4\mathfrak{m}_1}{\mathfrak{m}_0^2}\upepsilon_1 \mathds{E}(0)\int_S^T  \mathds{E}^{2\mu+1}\,dt+C(\upepsilon_1)
 \mathds{E}^{\mu+1}(S).
\end{split}
\end{equation}
Taking into account (\ref{L1}) into (\ref{L2})  and using (\ref{L3}), we obtain
\begin{equation}\label{P1}
\begin{split}
\int_S^T \mathds{E}^\mu\int_\Omega  \uptheta y_{tt}\,dx\,dt \leq 
C_2\upepsilon_1 \int_S^T  \mathds{E}^{2\mu+1}\,dt+C_2
 \mathds{E}^{\mu+1}(S)+ C_2\mathds{E}(0)\mathds{E}^{\mu+1}(S)
\end{split}
\end{equation}
Combining (\ref{G15}), (\ref{GG7}) and (\ref{P1}), we find
\begin{equation}\label{G1}
\begin{split}
2\int_S^T \mathds{E}^{\mu+1}\,dt
& \leq \ C\upepsilon\int_S^T  \mathds{E}^{2\mu+1}(t)\,dt+C \mathds{E}^{\mu+1}(S)+C \mathds{E}(0) \mathds{E}^{\mu+1}(S)+C\mathds{E}(S) 
\end{split}
\end{equation}
Choosing $\upepsilon$  small enough
and  $\mu = 0$, we obtain 
\begin{equation}\label{G1}
\begin{split}
\int_S^T\mathds{E}\,dt
& \leq C'\mathds{E}(S)+C' \mathds{E}(0)\mathds{E}(S)+C' \mathds{E}(S) 
\end{split}
\end{equation}
where $C'$ is  positive constant independent of $\mathds{E}(0)$.\\
Now, the proof is achieved by applying Lemma \ref{H1}.

\smallskip
\noindent
\hspace{0.2mm}\\\\
{\bf Acknowledgments}. The author would like to thanks to the editor
and anonymous referees for constructive comments and suggestions that improved
the quality of this manuscript.


\begin{thebibliography}{00}

\bibitem{adam}{R. A. Adams}, {\em Sobolev spaces}, Academic press, Pure and Applied
Mathematics, vol. {\bf 65}, (1978).

\bibitem{A}{P. Albano and D. Tataru}, {\em Carleman estimates and boundary observability for a coupled parabolic-hyperbolic system},  Electronic Journal of Differential  Equations, vol 2000, (2000), 1-15.
\bibitem{ak}{A. Ben Aissa, B. Gilbert and S. Nicaise}, {\em Same decay rate of second order evolution equations with or without
delay}, Systems \& Control Letters,  {\bf 141}, (2020), 104700.
\bibitem{B}{A. Benaissa and  A. Guesmia}, {\em Global existence and general decay estimates of solutions for degenerate or nondegenerate Kirchhoff equation with general dissipation}, J. Evol. Equ. {\bf 11} (2011), 1399-1424.
\bibitem{H1}{A. Haraux},\emph{Two remarks on dissipative hyperbolic problems}, in Lions, 
J. L. and Brezis, H. (Eds): Nonlinear Partial Differential Equations and 
Their Applications: College de France Seminar Volume XVIII (Research Notes 
in Mathematics Vol. \textbf{122}), Pitman: Boston, MA, (1985)
pp. 161-179.
\bibitem{KOM}{V. Komornik}, \emph{Exact Controllability and Stabilization. The Multiplier Method}, Masson Wiley, Paris (1994).
\bibitem{kirch}{G. Kirchhoff}, {\em Vorlesungen uber Mechanik, Teubner}, Leipzig, 1897.
\bibitem{Ma}{P. Martinez},{\em A new method to obtain decay rate estimates for dissipative systems}, ESAIM Control Optim. Calc. Var. {\bf 4} (1999) 419-444.
 \bibitem{Lio}{J.L. Lions},
{\em Quelques M\'ethodes De R\'esolution Des Probl\'emes Aux Limites Nonlin\'eaires, Dund Gautier-Villars}, Paris, 1969.

\bibitem{La2}{I. Lasiecka, S. Maad and A.  Sasane}, {\em Existence and Exponential Decay of Solutions to a Quasilinear Thermoelastic Plate System}, Nonlinear differ. equ. appl. {\bf15},  (2008), 689-715.

\bibitem{La1}{I. Lasiecka, M. Pokojovy and  X Wan},  {\em Long-time behavior of quasilinear thermoelastic Kirchhoff/Love
plates with second sound}, Nonlinear Analysis , vol. {\bf186},  (2019) 219-258.

 \bibitem{N1}{K. Nishihara, Y. Yamada}, {\em On global solutions of some degenerate quasilinear hyperbolic equations with dissipative terms}, Funkcial. Ekvac.
{\bf 33} (1) (1990) 151-159.
 
 \bibitem{O}{K. Ono}, {\em On global solutions and blow-up solutions of nonlinear Kirchhoff strings with nonlinear dissipation}, J. Math. Anal. Appl. {\bf 216} (1),(1997) 321-342.

 \bibitem{L1}{L. Tebou}, {\em Stabilization of some coupled hyperbolic/parabolic equations}, Discrete and continuous dynamical systems series B, Vol {\bf 14}, (4), (2010) 1601-1620.
 
 \bibitem{G}{G. Lebeau and E. Zuazua}, {\em Null-controllability of a system of linear thermoelasticity}, Arch.Rational Mech. Anal., {\bf141} (1998), 297-329.
 
 \bibitem{L2}{V. Keyantuo,  L. Tebou and M. Warma} {\em A gevrey class semigroup for a thermoelastic plate model with a fractional laplacien between the Euler-Bernoulli and Kirchhoff models}, Discrete and continuous dynamical systems series B, Vol
 {\bf 40}, (5), (2020), 2875-2889.
\bibitem{vill}{P. Villaggio}, {\em Mathematical Models for Elastic Structures}, Cambridge Univ. Press,  1997. 

\end{thebibliography}
\end{document}